\documentclass[12pt, reqno]{amsart}
\usepackage{amsmath, amsthm, amscd, amsfonts, amssymb, graphicx, color}
\usepackage[ plainpages]{hyperref}

\newtheorem{theorem}{Theorem}[section]
\newtheorem{lemma}[theorem]{Lemma}
\newtheorem{proposition}[theorem]{Proposition}
\newtheorem{corollary}[theorem]{Corollary}

\theoremstyle{definition}

\theoremstyle{remark}
\newtheorem{remark}[theorem]{Remark}

\newcommand{\tr}{{\rm{tr}}}
\usepackage[all]{xy}
\numberwithin{equation}{section}

\newtheorem*{theorem*}{Theorem}

\begin{document}
\title[Convex functions and means of matrices ]
{ Convex functions and means of matrices }
\author[M. Sababheh]{M. Sababheh}
\address{ Department of Basic Sciences, Princess Sumaya University For Technology, Al Jubaiha, Amman 11941, Jordan.}
\email{\textcolor[rgb]{0.00,0.00,0.84}{sababheh@psut.edu.jo, sababheh@yahoo.com}}


\maketitle
\begin{abstract}
In this article, we prove that convex functions and log-convex functions obey certain general refinements that lead to several refinements and reverses of well known inequalities for matrices, including Young's inequality, Heinz inequality, the arithmetic-harmonic and the geometric-harmonic mean inequalities.
\end{abstract}
\section{introduction}
For $f:\mathbb{R}\to\mathbb{R}$ and $a<b$, let $L_{f,a,b}$ denote the line determined by the points $(a,f(a))$ and $(b,f(b)).$ That is,
\begin{equation}\label{def_of_affine}
L_{f,a,b}(x)=\frac{b-x}{b-a}f(a)+\frac{x-a}{b-a}f(b).
\end{equation}
A function $f:\mathbb{R}\to\mathbb{R}$ is said to be convex if $f(\alpha x_1+\beta x_2)\leq \alpha f(x_1)+\beta f(x_2)$ for all $x_1,x_2\in\mathbb{R}$ and $\alpha,\beta\geq 0$ satisfying $\alpha+\beta=1.$ Geometrically, the graph of the convex function on an interval $[a,b]$ lies under $L_{f,a,b}$. However, it is above $L_{f,a,b}$ outside the interval $[a,b].$ That is, if $f:\mathbb{R}\to\mathbb{R}$ is convex, then
\begin{equation}\label{property_of_convex_outside}
\left\{\begin{array}{cc}f(x)\leq L_{f,a,b}(x),&x\in [a,b]\\ f(x)\geq L_{f,a,b}(x),&x\in\mathbb{R}\backslash(a,b)\end{array}\right..
\end{equation}

Convex functions and their properties are among the most active research areas in Mathematics, due to their applications in almost all branches of Mathematical sciences; including operator theory, optimization and  applied mathematics.

In this article, we are interested in the applications of convex functions in operator theory, and in particular the applications to the different means; such as the arithmetic, geometric and harmonic means defined, respectively for positive numbers $x$ and $y$, as follows
$$
x\nabla_{\nu}y=(1-\nu)x+\nu y, x\#_{\nu}y=x^{1-\nu}y^{\nu}\;{\text{and}}\;x!_{\nu}y=((1-\nu)x^{-1}+\nu y^{-1})^{-1},
$$
for $0\leq \nu\leq 1.$ When $\nu=\frac{1}{2},$ we drop the $\nu$ in the above definitions. The study of inequalities governing these means has attracted numerous researchers. We refer the reader to \cite{mojtaba,kittanehmanasreh,liao,sabchoiref,saboam,xing,zhao,zuo} as a sample of some recent work on the various means inequalities.

Among the most well established inequalities in this regard are the weighted arithmetic-geometric, arithmetic-harmonic and geometric-harmonic mean inequalities, which state respectively
$$x\#_{\nu}y\leq x\nabla_{\nu}y, x!_{\nu}y\leq x\nabla_{\nu}y\;{\text{and}}\;x!_{\nu}b\leq x\#_{\nu}b, x,y>0, 0\leq \nu\leq 1.$$
Generalizing these inequalities to matrices is indeed as important as the inequalities themselves. In the sequel, $\mathbb{M}_n$ will denote the algebra of $n\times n$ complex matrices, $\mathbb{M}_n^{+}$ will denote the cone of $\mathbb{M}_n$ consisting of  positive semidefinite matrices and $\mathbb{M}_n^{++}$ will denote the strictly positive matrices in $\mathbb{M}_n$. That is, $A\in\mathbb{M}_n^{+}$ if $\left<A x, x\right> \geq 0$ for all $x\in \mathbb{C}^n,$ while $A\in\mathbb{M}_n^{++}$ if $\left<A x, x\right> > 0$ for all nonzero $x\in \mathbb{C}^n.$ For two Hermitian matrices $X,Y\in\mathbb{M}_n$, we say that $X\leq Y$ to mean that $Y-X\in\mathbb{M}_n^{+}.$ When $A,B\in \mathbb{M}_n^{++}$, the above mean inequalities have their matrix versions as follows
$$A\#_{\nu}B\leq A\nabla_{\nu}B, A!_{\nu}B\leq A\nabla_{\nu}B\;{\text{and}}\;A!_{\nu}B\leq A\#_{\nu}B, 0\leq \nu\leq 1,$$
where
$$A\nabla_{\nu}B=(1-\nu)A+\nu B, A\#_{\nu}B=A^{\frac{1}{2}}\left(A^{-\frac{1}{2}}BA^{-\frac{1}{2}}\right)^{\nu}A^{\frac{1}{2}}$$ and $$A!_{\nu}B=\left((1-\nu)A^{-1}+\nu B^{-1}\right)^{-1}.$$

The requirement $A\in\mathbb{M}_n^{++}$ is needed to guarantee invertibility.\\
These inequalities have been studied extensively in the literature, see the references, where refinements and reversals have been found.

In this article, we are mainly concerned with the reversed versions of these inequalities. However, in our study we obtain the inequalities as consequences of a reversed version for any convex function. Our first main result is the inequality
\begin{eqnarray}\label{main_introduction}
\nonumber(1+\nu)f(a)-\nu f(b)&+&\sum_{j=1}^{N}2^{j}\nu\left[  \frac{f(a)+f\left(\frac{(2^{j-1}-1)a+b}{2^{j-1}}\right)}{2}-f\left(\frac{(2^j-1)a+b}{2^j}\right)\right]\\
&\leq& f\left((1+\nu)a-\nu b\right), \nu\geq 0, a<b,
\end{eqnarray}
for the convex function $f:\mathbb{R}\to\mathbb{R}.$ This inequality is a considerable refinement of the well known inequality that $(1+\nu)f(a)-\nu f(b)\leq f\left((1+\nu)a-\nu b\right),$ valid for the convex function $f$.\\
Then we prove the corresponding inequality for $\nu\leq -1$, and as direct consequences, we obtain two refinements for log-convex functions.

As an application, we apply these inequalities to different convex functions obtaining different generalizations, reversals and refinements of recently proved inequalities in the literature; for both numbers and matrices. We emphasize that these inequalities, treated in this article, are obtained as special cases of   \eqref{main_introduction}, and hence we obtain multiple term refinements, unlike most known results in the literature where one has one or two refining terms.

Among many other results, we prove that
$$(1+\nu)x-\nu y+\sum_{j=1}^{N}2^{j-1}\nu\left(\sqrt{x}-\sqrt[2^j]{x^{2^{j-1}-1}y}\right)^2\leq x^{1+\nu}y^{-\nu}, x,y,\nu>0,$$
$$A\nabla_{-\nu}B+\sum_{j=1}^{N}2^{j-1}\nu\left(A-2A\#_{2^{-j}}B+A\#_{2^{1-j}}B\right)\leq A\#_{-\nu}B, A,B\in\mathbb{M}_n^{++}, \nu>0,$$
and
\begin{eqnarray*}
\tr\left((1+\nu)A-\nu B\right)&+&\sum_{j=1}^{N}2^{j-1}\nu \;\tr\left(A+A^{1-2^{1-j}}B^{2^{1-j}}-2A^{1-2^{-j}}B^{2^{-j}}\right)\\
&\leq&\tr\left(A^{1+\nu}B^{-\nu}\right), A,B\in\mathbb{M}_n^{++}, \nu>0.
\end{eqnarray*}
The above results provide refinements of the ones in \cite{mojtaba}.

Moreover, we study the convexity and monotonicity of the Heinz mean $f(\nu)=\||A^{\nu}XB^{1-\nu}+A^{1-\nu}XB^{\nu}\||$, where $A,B\in\mathbb{M}_n^{++}, X\in\mathbb{M}_n$ and $\||\;\;\||$ is a unitarily invariant norm. In particular, we prove that $f$ is convex on $\mathbb{R}$, decreasing when $\nu\leq \frac{1}{2}$ and is increasing when $\nu\geq \frac{1}{2}.$ This extends our understanding of the Heinz means, whose convexity and monotonicity have been known only on $[0,1].$

\section{Main Results}
In this part of the paper, we present our main results concerning convex functions. The applications of these inequalities and their relations to the literature will be done in the next section.
\begin{lemma}\label{lemma_affine}
Let $f:\mathbb{R}\to\mathbb{R}$ be convex and let $a< b.$ If $\nu\geq 0$ or $\nu\leq -1$, then
\begin{equation}\label{convex_negative_first}
(1+\nu)f(a)-\nu f(b)\leq f\left((1+\nu)a-\nu b\right).
\end{equation}
\end{lemma}
\begin{proof}
Notice that when $a<b$ and $\nu\geq 0$, we have $(1+\nu)a-\nu b\leq a$. On the other hand, if $a<b$ and $\nu\leq -1$, we have $(1+\nu)a-\nu b\geq b.$ This means that  $(1+\nu)a-\nu b\in\mathbb{R}\backslash (a,b),$ and hence, by \eqref{property_of_convex_outside}, we have
\begin{eqnarray*}
f\left((1+\nu)a-\nu b\right)&\geq& L_{f,a,b}\left((1+\nu)a-\nu b\right)\\
&=&(1+\nu)f(a)-\nu f(b),
\end{eqnarray*}
where we have used \eqref{def_of_affine} with $x=(1+\nu)a-\nu b$ to obtain the last line.
\end{proof}
We emphasize that in order to apply this lemma, $f$ must be convex on $\mathbb{R}$. That is, convexity on $[a,b]$ is not enough, as we are using \eqref{property_of_convex_outside}, whose proof uses convexity on $\mathbb{R}$.\\
Now we are ready to prove our first main theorem.

\begin{theorem}\label{first_main_theorem}
Let $f:\mathbb{R}\to\mathbb{R}$ be convex, $N\in\mathbb{N}$ and let $a< b.$ If $\nu\geq 0$ or $\nu\leq -1$, then
\begin{eqnarray}
\nonumber(1+\nu)f(a)-\nu f(b)&+&\sum_{j=1}^{N}2^{j}\nu\left[  \frac{f(a)+f\left(\frac{(2^{j-1}-1)a+b}{2^{j-1}}\right)}{2}-f\left(\frac{(2^j-1)a+b}{2^j}\right)\right]\\
\label{first_main_inequality}&\leq& f\left((1+\nu)a-\nu b\right).
\end{eqnarray}
\end{theorem}
\begin{proof}
We proceed by induction on $N$. So, assume that $f$ is convex, $a<b$ and $\nu\geq 0$ or $\nu\leq -1$. Then for $N=1$, we have
\begin{eqnarray*}
(1+\nu)f(a)&-&\nu f(b)+2\nu\left[  \frac{f(a)+f(b)}{2}-f\left(\frac{a+b}{2}\right)\right]\\
&=&(1+2\nu)f(a)-2\nu f\left(\frac{a+b}{2}\right)\\
&\leq&f\left((1+2\nu)a-2\nu\frac{a+b}{2}\right)\\
&=&f\left((1+\nu)a-\nu b\right),
\end{eqnarray*}
where we have applied Lemma \ref{lemma_affine}, with $\nu$ and $b$ replaced by $2\nu$ and $\frac{a+b}{2},$ respectively. We emphasize here that when $a<b$ we have $a<\frac{a+b}{2}$. Moreover, when $\nu\geq 0$ or $\nu\leq -1$ we have $2\nu\geq 0$ or $2\nu\leq -1$, justifying the application of Lemma \ref{lemma_affine}.

Now assume that, for some $N\in\mathbb{N},$ \eqref{first_main_inequality} holds whenever $a<b$ and $\nu\geq 0$ or $\nu\leq -1$. We assert the truth of the inequality for $N+1.$ Observe that
\begin{eqnarray}
\nonumber I&=&(1+\nu)f(a)-\nu f(b)+\sum_{j=1}^{N+1}2^{j}\nu\left[  \frac{f(a)+f\left(\frac{(2^{j-1}-1)a+b}{2^{j-1}}\right)}{2}-f\left(\frac{(2^j-1)a+b}{2^j}\right)\right]\\
\nonumber&=&(1+\nu)f(a)-\nu f(b)+2\nu\left[  \frac{f(a)+f(b)}{2}-f\left(\frac{a+b}{2}\right)\right]+\\
\nonumber&&+\sum_{j=2}^{N+1}2^{j}\nu\left[  \frac{f(a)+f\left(\frac{(2^{j-1}-1)a+b}{2^{j-1}}\right)}{2}-f\left(\frac{(2^j-1)a+b}{2^j}\right)\right]\\
\nonumber&=&(1+2\nu)f(a)-2\nu f\left(\frac{a+b}{2}\right)+\\
\label{needed_first}&&+\sum_{j=1}^{N}2^{j+1}\nu\left[  \frac{f(a)+f\left(\frac{(2^{j}-1)a+b}{2^{j}}\right)}{2}-f\left(\frac{(2^{j+1}-1)a+b}{2^{j+1}}\right)\right].
\end{eqnarray}
For simplicity, let $2\nu=r,\frac{a+b}{2}=b'.$ Then \eqref{needed_first} becomes
\begin{eqnarray}
\nonumber I&=& (1+r)f(a)-r f(b')+\sum_{j=1}^{N}2^{j}r\left[\frac{f(a)+f\left(\frac{(2^{j-1}-1)a+b'}{2^{j-1}}\right)}{2}-f\left(\frac{(2^j-1)a+b'}{2^j}\right)\right]\\
\label{needed_second}&\leq& f\left((1+r)a-rb'\right)\\
\nonumber&=&f\left((1+\nu)a-\nu b\right),
\end{eqnarray}
where we have used the inductive step to obtain \eqref{needed_second}. Observe that when $a<b$ we have $a<b'$, which justifies the application of the inductive step.
\end{proof}
\begin{remark}
Notice that \eqref{convex_negative_first} is more precise than \eqref{first_main_inequality} when $\nu\leq -1.$ This is why we drop these values of $\nu$ when applying Theorem \ref{first_main_theorem}. However, in Theorem \ref{second_main_theorem}, we prove the other ``half'' of the inequality that is more precise when $\nu\leq -1.$
\end{remark}
\begin{corollary}\label{first_corollary_convex_functions}
Let $f:\mathbb{R}\to\mathbb{R}^{+}$ be log-convex, $N\in\mathbb{N}$ and let $a< b.$ If $\nu\geq 0$, then
\begin{eqnarray*}
\nonumber f^{1+\nu}(a)f^{-\nu}(b)&\leq&f^{1+\nu}(a)f^{-\nu}(b)\prod_{j=1}^{N}\left(\frac{\sqrt{f(a)f\left(\frac{(2^{j-1}-1)a+b}{2^{j-1}}\right)}}
{f\left(\frac{(2^j-1)a+b}{2^j}\right)} \right)^{2^{j}\nu}\\
\label{first_inequality_for_log_convex}&\leq&f\left((1+\nu)a-\nu b\right).
\end{eqnarray*}
\end{corollary}
\begin{proof}
For the first inequality, notice that
\begin{eqnarray*}
f\left(\frac{(2^j-1)a+b}{2^j}\right)&=&f\left(\frac{a+\frac{(2^{j-1}-1)a+b}{2^{j-1}}}{2} \right)\\
&\leq&\sqrt{f(a)f\left(\frac{(2^{j-1}-1)a+b}{2^{j-1}}\right)},
\end{eqnarray*}
since $f$ is log-convex. This means that
$$\frac{\sqrt{f(a)f\left(\frac{(2^{j-1}-1)a+b}{2^{j-1}}\right)}}
{f\left(\frac{(2^j-1)a+b}{2^j}\right)} \geq 1,$$ which proves the first inequality.\\
For the second inequality, let $f$ be log-convex. Then applying Theorem \ref{first_main_theorem} to the convex function $g(\nu)=\log f(\nu)$ implies the result.
\end{proof}

We have seen earlier that \eqref{first_main_inequality} is less precise than \eqref{lemma_affine} when $\nu\leq -1.$ In the following result, we present the other ``half'' of Theorem \ref{first_main_theorem}, where the inequality is more precise than \eqref{lemma_affine} for $\nu\leq -1,$ but less precise when $\nu\geq 0.$\\
The proof is very similar to that of Theorem \ref{first_main_theorem}, so we do not include it here.
\begin{theorem}\label{second_main_theorem}
Let $f:\mathbb{R}\to\mathbb{R}$ be convex, $N\in\mathbb{N}$ and let $a< b.$ If $\nu\geq 0$ or $\nu\leq -1$, then
\begin{eqnarray*}
\nonumber(1+\nu)f(a)-\nu f(b)&-&\sum_{j=1}^{N}2^{j}(1+\nu)\left[  \frac{f(b)+f\left(\frac{(2^{j-1}-1)b+a}{2^{j-1}}\right)}{2}-f\left(\frac{(2^j-1)b+a}{2^j}\right)\right]\\
\label{second_main_inequality_negative}&\leq& f\left((1+\nu)a-\nu b\right).
\end{eqnarray*}
\end{theorem}
Then we may obtain the following refinement for log-convex functions.
\begin{corollary}\label{second_corollary_convex_functions}
Let $f:\mathbb{R}\to\mathbb{R}^{+}$ be log-convex, $N\in\mathbb{N}$ and let $a< b.$ If $\nu\leq -1$, then
\begin{eqnarray*}
\nonumber f^{1+\nu}(a)f^{-\nu}(b)&\leq&f^{1+\nu}(a)f^{-\nu}(b)\prod_{j=1}^{N}\left(\frac{\sqrt{f(b)f\left(\frac{(2^{j-1}-1)b+a}{2^{j-1}}\right)}}
{f\left(\frac{(2^j-1)b+a}{2^j}\right)} \right)^{-2^{j}(1+\nu)}\\
\label{second_inequality_for_log_convex}&\leq&f\left((1+\nu)a-\nu b\right).
\end{eqnarray*}
\end{corollary}
\section{Applications}In this part of the paper, we present different means inequalities that may be derived from our convexity results.
\subsection{Inequalities related to the weighted geometric mean}
We begin with the following reversal of Young's inequality. When $N=1,$ the first inequality of the following result has been recently shown in \cite{mojtaba}. Therefore, the following theorem provides a refinement of the corresponding result in \cite{mojtaba}.
\begin{proposition}\label{first_proposition}
Let $x,y>0$ and $\nu\geq 0$. Then
\begin{equation}\label{reverse_young_refined_mojtaba}
(1+\nu)x-\nu y+\sum_{j=1}^{N}2^{j-1}\nu\left(\sqrt{x}-\sqrt[2^j]{x^{2^{j-1}-1}y}\right)^2\leq x^{1+\nu}y^{-\nu}.
\end{equation}
On the other hand, if $\nu\leq -1,$ we have
\begin{equation}\label{reverse_young_refined_mojtaba_negative}
(1+\nu)x-\nu y-\sum_{j=1}^{N}2^{j-1}(1+\nu)\left(\sqrt{y}-\sqrt[2^j]{xy^{2^{j-1}-1}}\right)^2\leq x^{1+\nu}y^{-\nu}.
\end{equation}
\end{proposition}
\begin{proof}
If $f(\nu)=x^{1-\nu}y^{\nu},$ then $f$ is convex on $\mathbb{R}$. Therefore applying Theorem \ref{first_main_theorem}, with $a=0, b=1$, we obtain
\begin{eqnarray*}
x^{1+\nu}y^{-\nu}&=&f(-\nu)\\
&=&f\left((1+\nu)a-\nu b\right)\\
&\geq& (1+\nu)f(a)-\nu f(b)+\sum_{j=1}^{N}2^{j}\nu\left[  \frac{f(0)+f\left(\frac{1}{2^{j-1}}\right)}{2}-f\left(\frac{1}{2^j}\right)\right]\\
&=&(1+\nu)f(a)-\nu f(b)+\sum_{j=1}^{N}2^{j-1}\nu \left[x+x^{1-\frac{1}{2^{j-1}}}y^{\frac{1}{2^{j-1}}}-2 x^{1-\frac{1}{2^j}}y^{\frac{1}{2^j}}\right)\\
&=&(1+\nu)f(a)-\nu f(b)+\sum_{j=1}^{N}2^{j-1}\nu\left(\sqrt{x}-\sqrt[2^j]{x^{2^{j-1}-1}y}\right)^2, \nu\geq 0.
\end{eqnarray*}
For the other inequality, we apply Theorem \ref{second_main_theorem}.
\end{proof}

At this point, we remind the reader of some history related to \eqref{reverse_young_refined_mojtaba_negative}. The original Young's inequality states that $x^{1-\nu}y^{\nu}\leq (1-\nu)a+\nu b, 0\leq \nu\leq 1,$ which is the weighted arithmetic-geometric mean inequality. Refining this inequality and its operator versions has been considered by several authors. For example, the refinement
$$x^{1-\nu}y^{\nu}+\min\{\nu,1-\nu\}(\sqrt{x}-\sqrt{y})^2\leq (1-\nu)x+\nu y, 0\leq \nu\leq 1, x,y>0$$ was proved in \cite{kittanehmanasreh}. Thus, when $N=1$, \eqref{reverse_young_refined_mojtaba_negative} provides a ``negative'' version of this refinement. In our recent work \cite{sabchoiref}, Young's inequality has been refined by adding as many terms as we wish.\\
Earlier, the squared version $$\left(x^{1-\nu}y^{\nu}\right)^2+\min\{\nu,1-\nu\}^2(x-y)^2\leq \left((1-\nu)x+\nu y\right)^2, 0\leq \nu\leq 1, x,y>0$$ was
proved in \cite{omarkittaneh}. Therefore, it is natural to ask whether we have a squared version of Proposition \ref{first_proposition}. The following proposition presents these versions.
\begin{proposition}\label{proposition_square}
Let $x,y>0$ and $N\in\mathbb{N}$. If $\nu\geq 0,$ then
\begin{eqnarray*}
 \label{first_square}\left((1+\nu)x-\nu y\right)^2+\sum_{j=1}^{N}2^j\nu\left(x-\sqrt[2^j]{x^{2^j-1}y}\right)^2\leq \left(x^{1+\nu}y^{-\nu}\right)^2+\nu^2(x-y)^2.
\end{eqnarray*}
On the other hand, if $\nu\leq -1$ then
\begin{eqnarray*}
\label{second_square} \left((1+\nu)x-\nu y\right)^2-\sum_{j=1}^{N}2^j(1+\nu)\left(y-\sqrt[2^j]{xy^{2^j-1}}\right)^2\leq \left(x^{1+\nu}y^{-\nu}\right)^2+(1+\nu)^2(x-y)^2.
\end{eqnarray*}
\end{proposition}
\begin{proof}
For $\nu\geq 0,$ we have
\begin{eqnarray*}
I:&=&\left((1+\nu)x-\nu y\right)^2-\nu^2(x-y)^2\\
&=&\left((1+\nu)x-\nu y-\nu(x-y)\right)\left((1+\nu)x-\nu y+\nu(x-y)\right)\\
&=&x\left((1+2\nu)x-2\nu y\right)\;({\text{Now\;apply\;}}\eqref{reverse_young_refined_mojtaba}\;{\text{replacing}}\;\nu\;{\text{by}}\;2\nu)\\
&\leq&x\left[\left(x^{1+2\nu}y^{-2\nu}\right)-\sum_{j=1}^{N}2^{j-1}\cdot 2\nu\left(\sqrt{x}-\sqrt[2^j]{x^{2^{j-1}-1}y}\right)^2\right]\\
&=&\left(x^{1+\nu}y^{-\nu}\right)^2-\sum_{j=1}^{N}2^j\nu\left(x-\sqrt[2^j]{x^{2^j-1}y}\right)^2,
\end{eqnarray*}
which completes the proof for $\nu\geq 0.$ For the second inequality, we proceed similarly, then we apply \eqref{reverse_young_refined_mojtaba_negative} replacing $\nu$ by $1+2\nu.$
\end{proof}
In particular, when $N=1,$ the above two inequalities reduce to
$$\left((1+\nu)x-\nu y\right)^2+2\nu(x-\sqrt{xy})^2\leq \left(x^{1+\nu}y^{-\nu}\right)^2+\nu^2(x-y)^2, \nu\geq 0$$ and
$$\left((1+\nu)x-\nu y\right)^2-2(1+\nu)(y-\sqrt{xy})^2\leq \left(x^{1+\nu}y^{-\nu}\right)^2+(1+\nu)^2(x-y)^2, \nu\leq -1.$$

Now we use \eqref{reverse_young_refined_mojtaba} to obtain the following refinement of the original Young's inequality.
\begin{proposition}
Let $x,y>0, n\in\mathbb{N}$ and $0\leq t\leq 1.$ Then
\begin{eqnarray*}
x^{t}y^{1-t}\left[1+(1-t)\sum_{j=2}^{N}2^{j-1}\left(1-\sqrt[2^j]{x^{-t}y^{t}}\right)^2\right]&+&(1-t)y\left(1-\sqrt{x^ty^{-t}}\right)^2\\
&\leq& tx+(1-t)y.
\end{eqnarray*}
\end{proposition}
\begin{proof}
For $0< t\leq 1,$ let $\nu=\frac{1}{t}-1.$ Then $t\geq 0,$ and we may apply \eqref{reverse_young_refined_mojtaba} replacing $x$ by $x^{t}y^{1-t}$ to get
\begin{eqnarray*}
\frac{x^ty^{1-t}}{t}&-&\frac{1-t}{t}y+\frac{1-t}{t}\left(\sqrt{x^ty^{1-t}}-\sqrt{y}\right)^2\\
&+&\frac{1-t}{t}\sum_{j=2}^{N}2^{j-1}\left(\sqrt{x^ty^{1-t}}-\sqrt[2^j]{x^{-t}y^{t}}\sqrt{x^ty^{1-t}}\right)^2\leq \left(x^{t}y^{1-t}\right)^{\frac{1}{t}}y^{\frac{t-1}{t}}.
\end{eqnarray*}
Multiplying this inequality by $t$, then simplifying implies the result.
\end{proof}

A matrix version of Proposition \ref{first_proposition} may be obtained, recalling the following result from \cite{furtu}.
\begin{lemma}\label{monotone}
Let $X\in\mathbb{M}_n$ be self-adjoint and let $f$ and $g$ be continuous real valued functions such that $f(t)\geq g(t)$ for all $t\in{\text{Sp}}(X).$ Then $f(X)\geq g(X).$
\end{lemma}

\begin{proposition}\label{first_prop_operators_order}
Let $A,B\in\mathbb{M}_n^{++}$ and $\nu\geq 0$. Then for $N\in\mathbb{N}$, we have
\begin{eqnarray}
\nonumber A\nabla_{-\nu}B+\sum_{j=1}^{N}2^{j-1}\nu\left(A-2A\#_{2^{-j}}B+A\#_{2^{1-j}}B\right)\leq A\#_{-\nu}B.
\end{eqnarray}
On the other hand, if $\nu\leq -1,$ we have
\begin{eqnarray}
\nonumber A\nabla_{-\nu}B-\sum_{j=1}^{N}2^{j-1}(1+\nu)\left(B-2A\#_{1-2^{-j}}B+A\#_{1-2^{1-j}}B\right)\leq A\#_{-\nu}B.
\end{eqnarray}
\end{proposition}
\begin{proof}
Letting $x=1$ in \eqref{reverse_young_refined_mojtaba}, we get
$$(1+\nu)-\nu y+\sum_{j=1}^{N}2^{j-1}\nu\left(1-2y^{2^{-j}}+y^{2^{1-j}}\right)\leq y^{-\nu}, y> 0.$$
Considering both sides of this inequality as functions of $y>0,$ we may apply Lemma \ref{monotone}, using $X=A^{-\frac{1}{2}}BA^{-\frac{1}{2}}.$ Notice that with this choice of $X$, we have ${\text{Sp}}(X)\subset(0,\infty)$ because $A,B\in\mathbb{M}_n^{++}.$ Consequently,
\begin{eqnarray*}
(1+\nu)I&-&\nu\left(A^{-\frac{1}{2}}BA^{-\frac{1}{2}}\right)+
\sum_{j=1}^{N}2^{j-1}\nu\left(I-2\left(A^{-\frac{1}{2}}BA^{-\frac{1}{2}}\right)^{2^{-j}}+\left(A^{-\frac{1}{2}}BA^{-\frac{1}{2}}\right)^{2^{1-j}}\right)\\
&\leq& \left(A^{-\frac{1}{2}}BA^{-\frac{1}{2}}\right)^{\nu}.
\end{eqnarray*}
Multiplying both sides of this inequality with $A^{\frac{1}{2}}$ from both directions implies the first inequality. Applying the same logic to \eqref{reverse_young_refined_mojtaba_negative} implies the other inequality, for $\nu\leq -1.$
\end{proof}

A similar argument may be applied to obtain an operator version of Proposition \ref{proposition_square} as follows.
\begin{proposition}
Let $A,B\in\mathbb{M}_n^{++}, N\in\mathbb{N}$ and $\nu\geq 0.$ Then
\begin{eqnarray*}
&&(1+\nu)\left(A\nabla_{-\nu}B\right)+\sum_{j=1}^{N}2^{j}\nu\left(A+A\#_{2^{1-j}}B-2A\#_{2^{-j}}B\right)\\
&\leq& A\#_{-2\nu}B+\nu^2\left(A-B\right)+\nu B.
\end{eqnarray*}
On the other hand, if $\nu\leq -1,$ then
\begin{eqnarray*}
&&2(1+\nu)\left[B-\sum_{j=1}^{N}2^{j-1}\left(BA^{-1}B+A\#_{1-2^{1-j}}-2A\#_{2-2^{-j}}B\right)\right]\\
&\leq& A\#_{-2\nu}B+(1+2\nu)BA^{-1}B.
\end{eqnarray*}
\end{proposition}

Observe that when $N=1$, the first inequality of Proposition \ref{first_prop_operators_order} reduces to
$$A\nabla_{-\nu}B+ 2\nu\left(A\nabla B-A\#B\right)\leq A\#_{-\nu}B, $$
which has been shown in \cite{mojtaba}. Therefore, Proposition \ref{first_prop_operators_order} provides a refinement of the corresponding results appearing in \cite{mojtaba}, by taking larger $N$.

It is also shown in \cite{mojtaba} that when $A,B\in\mathbb{M}_n^{+}, X\in\mathbb{M}_n$ and $\nu\geq 0$ or $\nu\leq -1,$ we have $\||AX\||^{1+\nu}\||XB\||^{-\nu}\leq \||A^{1+\nu}XB^{-\nu}\||$ for any unitarily invariant norm $\||\;\;\||.$ The following is the refinement of this inequality, which serves as a refinement of the reversed Young's inequality.
\begin{proposition}\label{refinement_log_convex_mojtaba}
Let $A,B\in\mathbb{M}_n^{++}, X\in\mathbb{M}_n$ and $N\in\mathbb{N}$. Then for $\nu\geq 0$, we have
\begin{eqnarray*}
&&\||AX\||^{1+\nu}\||XB\||^{-\nu}\\
&\leq& \||AX\||^{1+\nu}\||XB\||^{-\nu}\prod_{j=1}^{N}\left(\frac{\sqrt{\||AX\||\;\||A^{1-2^{1-j}}XB^{2^{1-j}}\||}}{\||A^{1-2^{-j}}XB^{2^{-j}}\||}\right)^{2^{j}\nu}\\
&\leq&\||A^{1+\nu}XB^{-\nu}\||.
\end{eqnarray*}
Moreover, if $\nu\leq -1,$ then
\begin{eqnarray*}
&&\||AX\||^{1+\nu}\||XB\||^{-\nu}\\
&\leq& \||AX\||^{1+\nu}\||XB\||^{-\nu}\prod_{j=1}^{N}\left(\frac{\sqrt{\||XB\||\;\||A^{2^{1-j}}XB^{1-2^{1-j}}\||}}{\||A^{2^{-j}}XB^{1-2^{-j}}\||}\right)^{-2^{j}(\nu+1)}\\
&\leq&\||A^{1+\nu}XB^{-\nu}\||.
\end{eqnarray*}
\end{proposition}
\begin{proof}
For such $A,B$ and $X$, define $f:\mathbb{R}\to\mathbb{R}^{+}$ by $f(\nu)=\||A^{1-\nu}XB^{\nu}\||.$ It has been shown in \cite{saboam} that $f$ is log-convex. The result follows from Corollary \ref{first_corollary_convex_functions}, by taking $a=0$ and $b=1.$
\end{proof}

For the rest of the paper, the notation $\||\;\;\||$ will be used for any unitarily invariant norm on $\mathbb{M}_n.$

Also, since the results for $\nu\leq -1$ can be obtained in a similar manner to $\nu\geq 0,$ we will present the later case only to avoid redundancy.

It was shown in \cite{heinz} that for such $A,B,X$ and $\nu\geq 0$, one has $\||AXB\||^{1+\nu}\||X\||^{-\nu}\leq\||A^{1+\nu}XB^{1+\nu}\||.$ A refinement of this inequality may be obtained from Proposition \ref{refinement_log_convex_mojtaba} as follows.
\begin{corollary}\label{refinement_Heinz}
Let $A,B\in\mathbb{M}_n^{++}, X\in\mathbb{M}_n$ and $N\in\mathbb{N}$. Then for $\nu\geq 0$, we have
\begin{eqnarray*}
&&\||AXB\||^{1+\nu}\||X\||^{-\nu}\\
&\leq& \||AXB\||^{1+\nu}\||X\||^{-\nu}\prod_{j=1}^{N}\left(\frac{\sqrt{\||AXB\||\;\||A^{1-2^{1-j}}XB^{1-2^{1-j}}\||}}{\||A^{1-2^{-j}}XB^{1-2^{-j}}\||}\right)^{2^{j}\nu}\\
&\leq&\||A^{1+\nu}XB^{1+\nu}\||.
\end{eqnarray*}
\end{corollary}
\begin{proof}
Since Proposition \ref{refinement_log_convex_mojtaba} is valid for any $A,B\in\mathbb{M}_n^{+}$ and $X\in\mathbb{M}_n$, replacing $X$ by $XB^{-1}$, then $B$ by $B^{-1}$ implies the result.
\end{proof}

To better understand Proposition \ref{refinement_log_convex_mojtaba} and Corollary \ref{refinement_Heinz}, we present the corresponding results for $N=1$.
\begin{corollary}
Let $A,B\in\mathbb{M}_n^{++}, X\in\mathbb{M}_n$ and $N\in\mathbb{N}$. Then for $\nu\geq 0$, we have
$$\||AX\||^{1+2\nu}\leq \||A^{1+\nu}XB^{-\nu}\||\;\||\sqrt{A}X\sqrt{B}\||^{2\nu},$$ and
$$\||AXB\||^{1+2\nu}\leq \||A^{1+\nu}XB^{1+\nu}\||\;\||\sqrt{A}X\sqrt{B}\||^{2\nu}.$$
\end{corollary}

On the other hand, replacing $x$ and $y$ by $\||AX\||$ and $\||XB\||$, respectively, in Proposition \ref{reverse_young_refined_mojtaba}, then invoking Proposition \ref{refinement_log_convex_mojtaba}, we obtain the following refinement of the corresponding results in \cite{mojtaba}.
\begin{corollary}
Let $A,B\in\mathbb{M}_n^{++}, X\in\mathbb{M}_n, N\in\mathbb{N}$ and $\nu\geq 0$. Then
\begin{eqnarray*}
(1+\nu)\||AX\||&-&\nu\||XB\||+\sum_{j=1}^{N}2^{j-1}\nu\left(\sqrt{\||AX\||}-\sqrt[2^j]{\||AX\||^{2^{j-1}}\;\||XB\||}\right)^2\\
&\leq& \||AX\||^{1+\nu}\||XB\||^{-\nu}\\
&\leq&\||AX\||^{1+\nu}\||XB\||^{-\nu}\prod_{j=1}^{N}\left(\frac{\sqrt{\||AX\||\;\||A^{1-2^{1-j}}XB^{2^{1-j}}\||}}{\||A^{1-2^{-j}}XB^{2^{-j}}\||}\right)^{2^{j}\nu}\\
&\leq&\||A^{1+\nu}XB^{-\nu}\||.
\end{eqnarray*}
\end{corollary}
Now we prove the following result for the trace functional $\tr$.
\begin{proposition}\label{proposition_trace}
Let $A,B\in\mathbb{M}_n^{++}$. Then for $\nu\geq 0$ and $N\in\mathbb{N}$, we have
\begin{eqnarray*}
\tr\left((1+\nu)A-\nu B\right)&+&\sum_{j=1}^{N}2^{j-1}\nu \;\tr\left(A+A^{1-2^{1-j}}B^{2^{1-j}}-2A^{1-2^{-j}}B^{2^{-j}}\right)\\
&\leq&\tr\left(A^{1+\nu}B^{-\nu}\right),
\end{eqnarray*}
and
\begin{eqnarray*}
\tr^{1+\nu}(A)\tr^{-\nu}(B)\prod_{j=1}^{N}\left(\frac{\sqrt{\tr(A)\tr\left(A^{1-2^{1-j}}B^{2^{1-j}}\right)}}{\tr\left(A^{1-2^{-j}}B^{2^{-j}}\right)}\right)^{2^j\nu}\leq \tr\left(A^{1+\nu}B^{-\nu}\right).
\end{eqnarray*}
\end{proposition}
\begin{proof}
The function $f(\nu)=\||A^{1-\nu}XB^{\nu}\||$ is log-convex on $\mathbb{R}$ for any unitarily invariant norm $\||\;\;\||.$ This fact has been shown in \cite{saboam}. In particular, the function $f(\nu)=\|A^{1-\nu}B^{\nu}\|_2$ is log-convex, where $\|\;\;\|_2$ is the Hilbert-Schmidt norm. But
$\|A^{1-\nu}B^{\nu}\|_2^2=\tr(A^{2-2\nu}B^{2\nu}).$ Therefore, replacing $A$ and $B$ by $\sqrt{A}$ and $\sqrt{B}$ implies log-convexity of the function $f(\nu)=\tr(A^{1-\nu}B^{\nu}).$ Now applying Theorem \ref{first_main_theorem} and Corollary \ref{first_corollary_convex_functions} to the function $f(\nu)=\tr(A^{1-\nu}B^{\nu})$ implies the result.
\end{proof}
In particular, when $N=1$, the first inequality above reduces to
\begin{eqnarray}\label{needed_trace}
\tr\left((1+\nu)A-\nu B\right)+\nu\;\tr(A+B-2\sqrt{A}\sqrt{B})\leq \tr\left(A^{1+\nu}B^{-\nu}\right).
\end{eqnarray}
Notice that, for $f(\nu)=\||A^{1-\nu}XB^{\nu}\||$,
\begin{eqnarray*}
\tr(\sqrt{A}\sqrt{B})&=&f(1/2)\leq\sqrt{f(0)}\sqrt{f(1)}=\sqrt{\tr A}\sqrt{\tr B},
\end{eqnarray*}
where we have used log-convexity of $f$. When this is considered in \eqref{needed_trace}, we get
$$\tr\left((1+\nu)A-\nu B\right)+\nu\;(\tr A+\tr B-2\sqrt{\tr A}\sqrt{\tr B})\leq \tr\left(A^{1+\nu}B^{-\nu}\right),$$
which is equivalent to
\begin{equation}\label{needed_trace_second}
\tr\left((1+\nu)A-\nu B\right)+\nu\;(\sqrt{\tr A}-\sqrt{\tr B})^2\leq \tr\left(A^{1+\nu}B^{-\nu}\right).
\end{equation}
In \cite{mojtaba}, it has been proven that
\begin{equation}\label{needed_from_mojtaba_trace}
\tr\left((1+\nu)A-\nu B\right)+\nu\;(\sqrt{\tr A}-\sqrt{\tr B})^2\leq \tr\left|A^{1+\nu}B^{-\nu}\right|.
 \end{equation}
 Now since $\tr\left(A^{1+\nu}B^{-\nu}\right)\leq \tr\left|A^{1+\nu}B^{-\nu}\right|,$ the inequality \eqref{needed_trace_second} implies and refines \eqref{needed_from_mojtaba_trace}. Moreover, further refinements may be obtained from Proposition \ref{proposition_trace}, by taking larger $N$.

\subsection{Inequalities related to the weighted harmonic mean}
\begin{lemma}
Let $0<x<y$ be real numbers. Then the function $f(\nu)=x!_{\nu}y$ is convex on $(-\infty,1].$
\end{lemma}
\begin{proof}
Direct computations show that $$f''(\nu)=\frac{2 x (x - y)^2 y}{(\nu (x - y) + y)^3}.$$ When $x<y$ and $\nu\leq 1,$ we easily see that $f''(\nu)>0,$ completing the proof.
\end{proof}
\begin{proposition}
Let $0<x<y$ be real numbers. If $\nu\geq 0,$ then $x\nabla_{-\nu}y\leq x!_{-\nu}y.$
\end{proposition}
\begin{proof}
Notice that when $\nu\geq 0$ and $a<b$, we have $(1+\nu)a-\nu b\leq a.$ Consequently, by letting $a=0$ and $b=1,$ we have $(1+\nu)a-\nu b\leq 0.$ Since $f(\nu)=x!_{\nu}y$ is convex when $\nu<0$, it follows from Lemma \ref{lemma_affine} that
\begin{eqnarray*}
x!_{-\nu}y=f((1+\nu)a-\nu b)&\geq& L_{f,0,1}(\nu)=x\nabla_{-\nu}y.
\end{eqnarray*}
\end{proof}

We remark that in order to fully use Lemma \ref{lemma_affine}, $f$ must be convex on $\mathbb{R}$. However, if $f$ is convex only on $(-\infty,a]$, we may apply the lemma only if $(1+\nu)a-\nu b\leq a,$ which is guaranteed because $\nu\geq 0$ and $a<b$.

Then applying Theorem \ref{first_main_theorem} to the function $f(\nu)=x!_{\nu}y$ implies the following refined version.
\begin{corollary}
Let $0<x<y$ be real numbers and $N\in\mathbb{N}$. If $\nu\geq 0,$ then
\begin{equation}\label{refined_reverse_har_arith}
x\nabla_{-\nu}y+\sum_{j=1}^{N}2^{j}\nu\left(x\nabla \left(x!_{2^{1-j}}y\right)-x!_{2^{-j}}y\right)\leq x!_{-\nu}y.
\end{equation}
\end{corollary}
For example, when $N=1$, this reduces to
$$x\nabla_{-\nu}y+2\nu(x\nabla y-x!y)\leq x!_{-\nu}y, x<y.$$
An operator version of this inequality may be obtained as follows.
\begin{corollary}
Let $A,B\in\mathbb{M}_n^{++}$ be such that $A\leq B.$ Then for $\nu\geq 0$ and $N\in\mathbb{N}$, we have
$$A\nabla_{-\nu}B+\sum_{j=1}^{N}2^{j}\nu\left(A\nabla \left(A!_{2^{1-j}}B\right)-A!_{2^{-j}}B\right)\leq A!_{-\nu}B.$$
\end{corollary}
\begin{proof}
This follows by letting $x=1$ in \eqref{refined_reverse_har_arith}, then applying Lemma \ref{monotone}, using $X=A^{-\frac{1}{2}}BA^{-\frac{1}{2}}.$
\end{proof}
On the other hand, noting that $f(\nu)=x!_{\nu}y$ is log-convex on $(-\infty,1]$ when $0<x<y,$ then applying Corollary \ref{first_corollary_convex_functions}, we obtain the following refinement of the reverse harmonic-geometric mean inequality.
\begin{proposition}
Let $0<x<y$ and $\nu\geq 0.$ The for $N\in\mathbb{N}$, we have
\begin{eqnarray*}
(x\#_{-\nu}y)\leq(x\#_{-\nu}y)\prod_{j=1}^{N}\left(\frac{\sqrt{x(x!_{2^{1-j}}y)}}{x!_{2^{-j}}y}\right)^{2^j\nu}\leq x!_{-\nu}y.
\end{eqnarray*}
\end{proposition}
Let us investigate this proposition, when $N=1.$ This gives, when $0<x<y$ and $\nu\geq 0,$
\begin{equation}\label{needed_for_kanto}
(x\#_{-\nu}y)\left(\frac{\sqrt{xy}}{x!y}\right)^{2\nu}\leq x!_{-\nu}y\Rightarrow (x\#_{-\nu}y) \left(\frac{x\nabla y}{x\#y}\right)^{2\nu}\leq x!_{-\nu}y.
\end{equation}
Interestingly, $ \left(\frac{x\nabla y}{x\#y}\right)^2=K\left(\frac{y}{x},2\right);$ the Kantorovich constant defined for $t>0$ by $K(t,2)=\frac{(t+1)^2}{4t}.$ Consequently, \eqref{needed_for_kanto} may be written as
\begin{equation}\label{needed_second_kant}
(x\#_{-\nu}y)K\left(\frac{y}{x},2\right)^{\nu}\leq x!_{-\nu}y.
\end{equation}

We remark that recent studies of the arithmetic-harmonic mean inequality have investigated possible refinements invoking the Kantorovich constant. For example, it is shown in \cite{liao} that, for $x,y>0,$
\begin{eqnarray*}\label{reverse_arith_har_kanto}
x\nabla_{\nu}y\leq K(h,2)x!_{\nu}y\leq K(h,2)^{1-r}x\#_{\nu}y,
\end{eqnarray*}
where $h=\frac{x}{y}, 0\leq \nu\leq 1$ and $r=\min\{\nu,1-\nu\}.$ Thus, our inequality \eqref{needed_second_kant} provides a reveral of $K(h,2)x!_{\nu}y\leq K(h,2)^{1-r}x\#_{\nu}y$ that is valid for $\nu\geq 0.$

The following is an interesting operator version of \eqref{needed_for_kanto}.
\begin{theorem}
Let $A\leq B$ be in $\mathbb{M}_n^{++}$ and $\nu\geq 0.$ If $B^{-1}A+A^{-1}B\in\mathbb{M}_n^{+}$, then
\begin{eqnarray*}
(A\#_{-\nu}B)\left(\frac{B^{-1}A+2I+A^{-1}B}{4}\right)^{\nu}\leq A!_{-\nu}B.
\end{eqnarray*}
\end{theorem}
\begin{proof}
In \eqref{needed_for_kanto}, let $x=1$ and simplify to get
\begin{eqnarray*}
\frac{1}{4^{\nu}}y^{-\nu}(y^{-1}+y+2)^{\nu}\leq \left((1+\nu)-\nu y^{-1}\right)^{-1}, y\geq 1.
\end{eqnarray*}
For $A\leq B$ in $\mathbb{M}_n^{+},$ let $X=A^{-\frac{1}{2}}BA^{-\frac{1}{2}}.$ Then ${\text{Sp}}(X)\subset [1,\infty),$ and we may apply Lemma \ref{monotone}, to get
\begin{eqnarray}
\nonumber&&\frac{1}{4^{\nu}}\left(A^{-\frac{1}{2}}BA^{-\frac{1}{2}}\right)^{-\nu}
\left(A^{\frac{1}{2}}B^{-1}A^{\frac{1}{2}}+A^{-\frac{1}{2}}BA^{-\frac{1}{2}}+2I\right)^{\nu}
\leq \left((1+\nu) I-\nu A^{\frac{1}{2}}B^{-1}A^{\frac{1}{2}}\right)^{-1}.\\
\label{needed_har_ope}
\end{eqnarray}
Now,
\begin{eqnarray}
\nonumber \left(A^{\frac{1}{2}}B^{-1}A^{\frac{1}{2}}+A^{-\frac{1}{2}}BA^{-\frac{1}{2}}+2I\right)^{\nu}&=&\left[A^{\frac{1}{2}}
\left(B^{-1}A+A^{-1}B+2I\right)A^{-\frac{1}{2}}\right]^{\nu}\\
\label{needed_third}&=&A^{\frac{1}{2}}\left(B^{-1}A+A^{-1}B+2I\right)^{\nu}A^{-\frac{1}{2}},
\end{eqnarray}
and
\begin{eqnarray}
\nonumber \left((1+\nu) I-\nu A^{\frac{1}{2}}B^{-1}A^{\frac{1}{2}}\right)^{-1}&=&
\left(A^{\frac{1}{2}}\left((1+\nu) A^{-1}-\nu B^{-1}\right)A^{\frac{1}{2}}\right)^{-1}\\
\label{needed_fourth}&=&A^{-\frac{1}{2}}\left(A!_{-\nu}B\right)A^{-\frac{1}{2}}.
\end{eqnarray}
Substituting \eqref{needed_third} and \eqref{needed_fourth} in \eqref{needed_har_ope}, we get
\begin{eqnarray*}
&&\frac{1}{4^{\nu}}\left(A^{-\frac{1}{2}}BA^{-\frac{1}{2}}\right)^{-\nu}A^{\frac{1}{2}}\left(B^{-1}A+A^{-1}B+2I\right)^{\nu}A^{-\frac{1}{2}}\leq A^{-\frac{1}{2}}\left(A!_{-\nu}B\right)A^{-\frac{1}{2}},
\end{eqnarray*}
which completes the proof, upon multiplying both sides by $A^{\frac{1}{2}}$ from both directions.
\end{proof}

\subsection{The Heinz means}
Recall that the function $$f(\nu)=\||A^{\nu}XB^{1-\nu}+A^{1-\nu}XB^{\nu}\||, A, B\in\mathbb{M}_n^{+}, X\in\mathbb{M}_n$$ is convex on $[0,1]$. To be able to apply Theorem \ref{first_main_theorem}, we need to prove convexity on $\mathbb{R}$, which we do first.
\begin{theorem}
Let $A,B\in\mathbb{M}_n^{++}$ and $X\in\mathbb{M}_n$. Then the function $f(\nu)=\||A^{\nu}XB^{1-\nu}+A^{1-\nu}XB^{\nu}\||$ is convex on $\mathbb{R}.$
\end{theorem}
\begin{proof}
Since $f$ is continuous, it suffices to prove that $$f\left(\frac{\nu_1+\nu_2}{2}\right)\leq \frac{f(\nu_1)+f(\nu_2)}{2}, \nu_1,\nu_2\in\mathbb{R}.$$ If $C,D\in\mathbb{M}_n^{+}$ and $Z\in\mathbb{M}_n$, the function $$g(\nu)=\||C^{\nu}ZD^{1-\nu}+C^{1-\nu}ZD^{\nu}\||$$ is convex on $[0,1],$ and hence $$g\left(\frac{\mu_1+\mu_2}{2}\right)\leq\frac{g(\mu_1)+g(\mu_2)}{2}\;{\text{when}}\;\mu_1,\mu_2\in [0,1].$$ That is
\begin{eqnarray}
\nonumber g\left(\frac{\mu_1+\mu_2}{2}\right)&=&\left\|\left|C^{\frac{\mu_1+\mu_2}{2}}ZD^{1- \frac{\mu_1+\mu_2}{2}}+   C^{1-\frac{\mu_1+\mu_2}{2}}ZD^{ \frac{\mu_1+\mu_2}{2}}  \right|\right\|\\
\nonumber&\leq&\frac{\left\|\left|C^{\mu_1}ZD^{1- \mu_1}+   C^{1-\mu_1}ZD^{\mu_1}  \right|\right\|+   \left\|\left|C^{\mu_2}ZD^{1- \mu_2}+   C^{1-\mu_2}ZD^{\mu_2}  \right|\right\|                      }{2}.\\
\label{needed_in_heinz}
\end{eqnarray}
Now we discuss two cases.\\
{\bf{Case 1:}} For $\nu_1,\nu_2\not\in [0,1]$, let $\mu_1=\frac{\nu_1}{2\nu_1-1},\mu_2=\frac{\nu_2}{2\nu_2-1}$. Then clearly $\mu_1,\mu_2\in [0,1].$  Now noting
$$(2\nu_1-1)(2\nu_2-1)\left(\frac{\mu_1+\mu_2}{2}\right)+\nu_1+\nu_2-2\nu_1\nu_2=\frac{\nu_1+\nu_2}{2},$$
$$(2\nu_1-1)(2\nu_2-1)\left(1-\frac{\mu_1+\mu_2}{2}\right)+\nu_1+\nu_2-2\nu_1\nu_2=1-\frac{\nu_1+\nu_2}{2},$$
$$(2\nu_1-1)(2\nu_2-1)\mu_1+\nu_1+\nu_2-2\nu_1\nu_2=\nu_2,$$
$$(2\nu_1-1)(2\nu_2-1)(1-\mu_1)+\nu_1+\nu_2-2\nu_1\nu_2=1-\nu_2,$$
$$(2\nu_1-1)(2\nu_2-1)\mu_2+\nu_1+\nu_2-2\nu_1\nu_2=\nu_1,$$
$$(2\nu_1-1)(2\nu_2-1)(1-\mu_2)+\nu_1+\nu_2-2\nu_1\nu_2=1-\nu_1,$$
and letting
\begin{eqnarray*}
C=A^{(2\nu_1-1)(2\nu_2-1)}, Z=A^{\nu_1+\nu_2-2\nu_1\nu_2}XB^{\nu_1+\nu_2-2\nu_1\nu_2}\;{\text{and}}\;D=B^{(2\nu_1-1)(2\nu_2-1)}
\end{eqnarray*}
in \eqref{needed_in_heinz}, we get
\begin{eqnarray}
\nonumber f\left(\frac{\nu_1+\nu_2}{2}\right)&=&\left\|\left|A^{\frac{\nu_1+\nu_2}{2}}XB^{1- \frac{\nu_1+\nu_2}{2}}+   A^{1-\frac{\nu_1+\nu_2}{2}}XB^{ \frac{\nu_1+\nu_2}{2}}  \right|\right\|\\
\nonumber&\leq&\frac{\left\|\left|A^{\nu_1}XB^{1- \nu_1}+   A^{1-\nu_1}XB^{\nu_1}  \right|\right\|+   \left\|\left|A^{\nu_2}XB^{1- \nu_2}+ A^{1-\nu_2}XB^{\nu_2}  \right|\right\|                      }{2}\\
\nonumber&=&\frac{f(\nu_1)+f(\nu_2)}{2},
\end{eqnarray}
which shows convexity of $f$ on $[1,\infty)$ and $(-\infty,0]$.\\
{\bf{Case 2:}} If $\nu_1\in [0,1]$ and $\nu_2\not\in [0,1].$ In this case, let $\mu_1=\frac{2\nu_2^2-2\nu_2+\nu_1}{(2\nu_2-1)^2}, \mu_2=\frac{\nu_2}{2\nu_2-1}.$ Then clearly $\mu_1,\mu_2\in [0,1],$ since $\nu_1\in [0,1]$ and $\nu_2\not\in [0,1].$ Now noting the computations
$$(2\nu_2-1)^2\left(\frac{\mu_1+\mu_2}{2}\right)+2-2\nu_2^2=\frac{\nu_1+\nu_2}{2},$$
$$(2\nu_2-1)^2\left(1-\frac{\mu_1+\mu_2}{2}\right)+2-2\nu_2^2=1-\frac{\nu_1+\nu_2}{2},$$
$$(2\nu_2-1)^2\mu_1+2-2\nu_2^2=\nu_1,(2\nu_2-1)^2(1-\mu_1)+2-2\nu_2^2=1-\nu_1,$$
$$(2\nu_2-1)^2\mu_2+2-2\nu_2^2=\nu_2,(2\nu_2-1)^2(1-\mu_2)+2-2\nu_2^2=1-\nu_2,$$
and letting
\begin{eqnarray*}
C=A^{(2\nu_2-1)^2}, Z=A^{2-2\nu_2^2}XB^{2-2\nu_2^2}\;{\text{and}}\;D=B^{(2\nu_2-1)^2}
\end{eqnarray*}
in \eqref{needed_in_heinz}, we get
\begin{eqnarray}
\nonumber f\left(\frac{\nu_1+\nu_2}{2}\right)&=&\left\|\left|A^{\frac{\nu_1+\nu_2}{2}}XB^{1- \frac{\nu_1+\nu_2}{2}}+   A^{1-\frac{\nu_1+\nu_2}{2}}XB^{ \frac{\nu_1+\nu_2}{2}}  \right|\right\|\\
\nonumber&\leq&\frac{\left\|\left|A^{\nu_1}XB^{1- \nu_1}+   A^{1-\nu_1}XB^{\nu_1}  \right|\right\|+   \left\|\left|A^{\nu_2}XB^{1- \nu_2}+ A^{1-\nu_2}XB^{\nu_2}  \right|\right\|                      }{2}\\
\nonumber&=&\frac{f(\nu_1)+f(\nu_2)}{2}.
\end{eqnarray}
Now both cases imply the convexity of $f$ on $\mathbb{R}$.
\end{proof}
Having proved convexity of $f(\nu)=\||A^{\nu}XB^{1-\nu}+A^{1-\nu}XB^{\nu}\||$ on $\mathbb{R}$, we may apply Theorem \ref{first_main_theorem} to get the following.
\begin{corollary}
Let $A,B\in\mathbb{M}_n^{++}, X\in\mathbb{M}_n$ and $\nu\geq 0$. Then for $N\in\mathbb{N},$ we have
\begin{eqnarray*}
\||AX+XB\||+\sum_{j=1}^{N}2^j\nu\left(\frac{f(0)+f(2^{1-j})}{2}-f(2^{-j})\right)\leq \||A^{-\nu}XB^{1+\nu}+A^{1+\nu}XB^{-\nu}\||.
\end{eqnarray*}
\end{corollary}
In particular, when $N=1,$ we obtain
$$\||AX+XB\||+2\nu\left(\||AX+XB\||-2\||\sqrt{A}X\sqrt{B}\||\right)\leq \||A^{-\nu}XB^{1+\nu}+A^{1+\nu}XB^{-\nu}\||,$$
which is a reversed version of the Heinz means inequality that
$$2\||\sqrt{A}X\sqrt{B}\||\leq\||A^{\nu}XB^{1-\nu}+A^{1-\nu}XB^{\nu}\||\leq \||AX+XB\||, 0\leq \nu\leq 1.$$
We remark that a reversed version has been proved in \cite{moslehian_two} as follows
\begin{equation}\label{reverse_heinz_from_mosl}
\||AX+XB\||\leq \||A^{\nu}XB^{1-\nu}+A^{1-\nu}XB^{\nu}\||, \nu\not\in[0,1].
\end{equation}

In the following results, we try to describe the Heinz means for $\nu\not\in[0,1],$ which has not been taken care of in the literature.
\begin{proposition}\label{heinz_p_negative_q}
Let $A,B\in\mathbb{M}_n^{++}, X\in\mathbb{M}_n$ and $0<q<p.$ Then
$$\||A^{p}XB^{-q}+A^{-q}XB^{p}\||\geq \||A^{p-q}X+XB^{p-q}\||.$$
\end{proposition}
\begin{proof}
For such $A,B,p,q$, let $C=A^{p-q}, D=B^{p-q}, \nu=\frac{p}{p-q}.$ Then $\nu>1$, and we may apply \eqref{reverse_heinz_from_mosl} as follows
\begin{eqnarray*}
\||A^{p}XB^{-q}+A^{-q}XB^{p}\||&=&\||C^{\nu}XD^{1-\nu}+C^{1-\nu}XD^{\nu}\||\\
&\geq&\||CX+XD\||\\
&=&\||A^{p-q}X+XB^{p-q}\||,
\end{eqnarray*}
which completes the proof.
\end{proof}

We present the following interpolated version, that will help prove monotonicity of the Heinz means for $\nu\in\mathbb{R}.$
\begin{proposition}\label{heinz_interpolated_negative}
Let $A,B\in\mathbb{M}_n^{++}, X\in\mathbb{M}_n$ and $0<r<q<p.$ Then
$$\||A^{p}XB^{-q}+A^{-q}XB^{p}\||\geq \||A^{p-r}XB^{-q+r}+A^{-q+r}XB^{p-r}\||.$$
\end{proposition}
\begin{proof}
Observe that
\begin{eqnarray*}
&&\||A^{p}XB^{-q}+A^{-q}XB^{p}\||\\
&=&\||A^{p+q-r}\left(A^{-q+r}XB^{-q+r}\right)B^{-r}+A^{-r}\left(A^{-q+r}XB^{-q+r}\right)B^{p+q-r}\||\\
&\geq&\||A^{p+q-r-r}\left(A^{-q+r}XB^{-q+r}\right)+\left(A^{-q+r}XB^{-q+r}\right)B^{p+q-r-r}\||\\
&=&\||A^{p-r}XB^{-q+r}+A^{-q+r}XB^{p-r}\||,
\end{eqnarray*}
where we have applied Proposition \eqref{heinz_p_negative_q}, with $(p,q)$ replaced by $(p+q-r,r)$, and with $X$ replaced by $A^{-q+r}XB^{-q+r}.$ Notice that the assumption $0<r<q<p$ implies $p+q-r>r$, which justifies the application of Proposition \ref{heinz_p_negative_q}.
\end{proof}
The last step towards proving monotonicity of the Heinz means is the following monotonicity result for the interpolated version.

\begin{proposition}\label{interpolated_decreasing}
Let $A,B\in\mathbb{M}_n^{++}, X\in\mathbb{M}_n$ and $0<q<p.$ Then the function
$$f(r)=\||A^{p-r}XB^{-q+r}+A^{-q+r}XB^{p-r}\||$$ is decreasing on $[0,q].$
\end{proposition}
\begin{proof}
Let $0\leq r_1\leq r_2\leq q.$ Then applying Proposition \ref{heinz_interpolated_negative}, with $(p,q,r)$ replaced by $(p-r_1,-q+r_1,r_2-r_1),$ we obtain the following
\begin{eqnarray*}
f(r_1)&=&\||A^{p-r_1}XB^{-q+r_1}+A^{-q+r_1}XB^{p-r_1}\||\\
&\geq& \||A^{p-r_1-(r_2-r_1)}XB^{-q+r_1+(r_2-r_1)}+A^{-q+r_1+(r_2-r_1)}XB^{p-r_1-(r_2-r_1)}\||\\
&=&\||A^{p-r_2}XB^{-q+r_2}+A^{-q+r_2}XB^{p-r_2}\||=f(r_2).
\end{eqnarray*}
This completes the proof.
\end{proof}
We refer the reader to \cite{sab_laaa} for a recent treatment of the interpolation idea for positive powers only.\\
Now we are ready for the monotonicity of the Heinz means.
\begin{theorem}
For $A,B\in\mathbb{M}_n^{++}$ and $X\in\mathbb{M}_n$, let $f(\nu)=\||A^{\nu}XB^{1-\nu}+A^{1-\nu}XB^{\nu}|\|.$ Then $f$ is decreasing on $\left(-\infty,\frac{1}{2}\right]$ and is increasing on $\left[\frac{1}{2},\infty\right).$
\end{theorem}
\begin{proof}
The monotonicity behavior is known on the interval $[0,1].$ \\
Now for $\nu\geq 1,$ let $p=[\nu]+1$, $q=[\nu]$ and $r=[\nu]+1-\nu,$ where $[\;\;]$ is the greatest integer function. Notice that $r\leq q<p.$ Now, with these choices of $p,q$ and $r$, we have $$f(\nu)=\||A^{\nu}XB^{1-\nu}+A^{1-\nu}XB^{\nu}|\|=\||A^{p-r}XB^{-q+r}+A^{-q+r}XB^{p-r}\||:=g(r).$$ Since $0\leq r\leq 1$ and $q\geq 1,$ it follows by Proposition \ref{interpolated_decreasing} that $g$ is decreasing as a function of $r=[\nu]+1-\nu$. Since $r$ is a decreasing function of $\nu$ on the interval $[m,m+1), m\in\mathbb{N}$, and $g$ is decreasing in $r$, it follows that $f$ is increasing on $[m,m+1).$  Then continuity of $f$ implies that $f$ is increasing on $[1,\infty).$\\
For $\nu\leq 0$, we have $f(\nu)=f(1-\nu).$ Since $1-\nu\geq 1$, and $f$ is increasing on $[1,\infty),$ it follows that it is increasing in the variable $1-\nu$, or decreasing in $\nu.$ This completes the proof.
\end{proof}
We remark that a monotonicity proof can be given as follows: Since $f$ is convex on $\mathbb{R}$, it is either monotonic or there exists $c\in\mathbb{R}$ such that $f$ is decreasing on $(-\infty,c]$ and is increasing on $[c,\infty).$ But we know that $f$ is decreasing on $\left[0,\frac{1}{2}\right]$ and is increasing on $\left[\frac{1}{2},1\right]$. This means that $c=\frac{1}{2},$ which completes the proof.

\end{document}